\documentclass{article}
\usepackage[utf8]{inputenc}
\usepackage{amsfonts}
\usepackage{amsthm}
\usepackage{amssymb,latexsym}
\usepackage{amsmath}
\usepackage{amssymb}
\usepackage{cite}
\usepackage{physics}
\usepackage{authblk}
\usepackage[margin=1.2in]{geometry}
\usepackage{listings}
\usepackage{xcolor}
\usepackage{hyperref}
\usepackage{nccmath}
\usepackage{multicol}
\usepackage{mathtools}
\usepackage{graphicx}
\usepackage{verbatim} 
\usepackage{float}
\usepackage{subcaption}
\usepackage{tikz}
\usepackage{pgfplots}
\usepackage[justification=centering]{caption}
\graphicspath{ {./images/} }

\newcommand{\Z}{\mathbb{Z}}
\newcommand{\Q}{\mathbb{Q}}
\newcommand{\C}{\mathbb{C}}
\newcommand{\dsum}{S_{\chi_1, \chi_2}}

\newcommand{\oo}[1]{\overline{#1}}
\newcommand{\abd}{\begin{bmatrix}
        a & b \\
        0 & d
     \end{bmatrix}}
\newcommand{\abdp}{\begin{bmatrix}
        a' & b' \\
        0 & d'
     \end{bmatrix}}
\newcommand{\mat}{\begin{bmatrix}
        a & b \\
        c & d
     \end{bmatrix}}   
\newcommand{\hkl}{\begin{bmatrix}
        h & * \\
        k & l
     \end{bmatrix}}

\newcommand{\projsum}{\frac{1}{\phi(N)} \sum_{j \in (\Z/N\Z)^{\cross}}}
\newcommand{\Mod}[1]{\ (\mathrm{mod}\ #1)}

\newcommand{\Addresses}{{
  \bigskip
  \footnotesize

  M.~Majure, \textsc{Department of Mathematics, University of Iowa,
    Iowa City, Iowa 52240}\par\nopagebreak
  \textit{E-mail address}, M.~Majure: \texttt{majuremitch@gmail.com}
}}

\DeclareMathOperator{\sgn}{sgn}

\newtheorem{theorem}{Theorem}[section]
\newtheorem{conjecture}[theorem]{Conjecture}
\newtheorem{corollary}[theorem]{Corollary}
\newtheorem{definition}[theorem]{Definition}
\newtheorem{lemma}[theorem]{Lemma}
\newtheorem{proposition}[theorem]{Proposition}
\newtheorem*{remark}{Remark}
\theoremstyle{definition}

\title{Algebraic properties of the values of newform Dedekind sums}
\author{Mitchell Majure}
\date{July 2022}

\begin{document}

\maketitle

\begin{abstract}
     We study the image of a generalized Dedekind sum relating to the weight zero Eisenstein series $E_{\chi_1,\chi_2}$. We show that the image is a lattice of full rank inside a number field determined by the characters $\chi_1$ and $\chi_2$. We also give a generalization of Knopp's identity for the classical Dedekind sum.
\end{abstract}


\section{Introduction}
For coprime integers $h, k$ where $k>0$, the classical Dedekind sum is defined as
$$s(h,k) = \sum_{j \Mod{k}}B_1\bigg(\frac{j}{k}\bigg)B_1\bigg(\frac{hj}{k}\bigg),$$\\
where $B_1(x)$ is the first Bernoulli function (also known as the sawtooth function):
\begin{equation*}
    B_1(x) = 
    \begin{cases}
        0, & \text{if } x\in\Z\\
        x-\lfloor x \rfloor-\frac{1}{2} & \text{otherwise.}
    \end{cases}
\end{equation*}
The classical Dedekind sum was first introduced to study the automorphy factor for the transformation of the Dedekind $\eta$ function. It has also appeared outside of number theory, where a particularly fascinating example is in the enumeration of lattice points in tetrahedra. One can find a very thorough discussion of the Dedekind sum's properties in \cite{sums}.\\

Many papers have investigated the values taken by the classical Dedekind sum. In \cite{Hickerson}, it is shown that the values are dense in $\mathbb{R}.$ There is a standing conjecture of Girstmair which would completely determine the values of the normalized Dedekind sum $12s(h,k)$. Specifically,
\begin{conjecture}[\cite{Girst}]
For a natural number $q \geq 2$, and $k \in \Z$ coprime to $q$, $\frac{k}{q}$ is a value of the normalized Dedekind sum if, and only if, the following hold:
\begin{enumerate}
    \item If $3 \nmid q$, then $k \equiv 0 \pmod{3}$.
    \item If $2 \nmid q$, then
    $k \equiv \begin{cases}
        2 \pmod{4} & \text{if } q \equiv 3 \pmod{4};\\
        0 \pmod{8} & \text{if } q \text{ is a square;}\\
        0 \pmod{4} & \text{otherwise.}
    \end{cases}$
\end{enumerate}
\end{conjecture}

There are several generalizations of this sum in the literature. In particular we study the Dedekind sum associated to the Eisenstein series with two primitive Dirichlet characters discussed in \cite{SVY}, \cite{DG}, \cite{NRY}, and \cite{LBY}. Both \cite{NRY} and \cite{LBY} study the kernel of these Dedekind sums, which has apparently avoided simple characterization. A complementary aspect not covered in these papers is the image of the newform sums. In this paper we are able to determine the structure of the image of the newform Dedekind sum and the number field in which it lies. In addition, we generalize an identity of the classical Dedekind sum and give a brief discussion on the cohomological aspect of the newform sums. \\

We introduce the newform Dedekind sum by its finite sum definition.
\begin{definition} [\cite{SVY}]
\label{def:finite}
    Let $\chi_1,\chi_2$ be primitive Dirichlet characters modulo $q_1$ and $q_2$ (respectively) such that $\chi_1\chi_2(-1) = 1$ and $q_1,q_2 > 1$. Let $\gamma = \hkl \in \Gamma_0(q_1q_2)$ such that $k \geq 1$. Then
    $$\dsum(\gamma) = \dsum(h,k)=\sum_{j \Mod{k}}\sum_{n \Mod{q_1}}\oo{\chi_2}(j)\oo{\chi_1}(n)B_1\bigg(\frac{j}{k}\bigg)B_1\bigg(\frac{n}{q_1}+\frac{hj}{k}\bigg).$$
\end{definition}
\begin{definition} 
Denote by $F_{\chi_1,\chi_2}$, the smallest number field in which $\chi_1$ and $\chi_2$ take values.
\end{definition}
The newform Dedekind sums exhibit a wealth of properties. One that is both basic and highly important is the crossed homomorphism property. 

\begin{proposition}[Crossed Homomorphism Property, \cite{SVY}]
\label{crossed}
    For $\gamma_1,\gamma_2\in\Gamma_0(q_1q_2)$
    $$\dsum(\gamma_1\gamma_2) = \dsum(\gamma_1)+\psi(\gamma_1)\dsum(\gamma_2).$$
    We call $\psi = \chi_1\oo{\chi_2}$ the central character of $\dsum$.
\end{proposition}
\begin{remark}
This is to say $\dsum$ is an element of the space $H^1(\Gamma_0(N),\C^{\psi})$ (see Proposition \ref{decomp} and Section \ref{sec:discussion} for discussion). When $\psi=\boldsymbol{1},\,\dsum\in\text{Hom}(\Gamma_0(q_1q_2),\C).$ By restriction, we always have  $\dsum\in\text{Hom}(\Gamma_1(q_1q_2),\C).$
\end{remark}
Our primary result is a description of the structure of $\dsum(\Gamma_1(q_1q_2))$.
\begin{theorem}
\label{thm:structure}
    The image $\dsum(\Gamma_1(q_1q_2))$ is a lattice (of full rank) inside $F_{\chi_1,\chi_2}$.
\end{theorem}
The reader may wonder how we can describe the image of the newform Dedekind sum, while the classical case still remains open. 
This is a consequence of the differences between $SL_2(\Z)$ and the congruence subgroup $\Gamma_1(N)$. The former has no non-trivial homomorphisms into $\C$, while the latter does with the newform Dedekind sums. We also note this deviates from the case of modular forms, where more complication arises when restricting focus to congruence subgroups.

The newform sum also satisfies a generalization of an identity due to Knopp \cite{Knopp}:
    \begin{proposition}
    \label{prp:knopp}
        For $h,k,n \in \Z,\,k,n>0,$
        $$\sum_{ad=n}\,\sum_{b \Mod{d}}s(ah+bk,dk) = \sigma(n)s(h,k), \:\: \, \sigma(n) = \sum_{d|n}d.$$
    \end{proposition}
Knopp proves this by the action of the Hecke operator $T_n$ on $\text{log}(\eta)$. Elementary proofs exist, such as \cite{elem}, using only the arithmetic properties of the classical Dedekind sum and $B_1(x)$. The newform identity incroporates a twist by the central character of $\dsum$.
    \begin{theorem}[Generalized Knopp Identity]
    \label{thm:newknopp}
    For $h,k,n \in \Z,\, q_1q_2|k$, and $n,k>0$
        \begin{equation*}
        \sum_{\substack{ad=n\\(a,n)=1}}\psi(a)\sum_{b \Mod{d}}\dsum(ah+bk,dk) = \rho_{\chi_1,\chi_2}(n)\dsum(h,k), \end{equation*} 
        where 
        \begin{equation*}
         \rho_{\chi_1,\chi_2}(n) = \sum_{d|n}\chi_1\bigg(\frac{n}{d}\bigg)\overline{\chi_2}(d)d.
        \end{equation*}
    \end{theorem}
\noindent This identity is not only beautiful, but will also allow us to deduce the following:
\begin{proposition}
\label{independence}
The newform Dedekind sums definded in \ref{def:finite} are linearly independent in $Hom(\Gamma_1(q_1q_2),\C)$.
\end{proposition}
\begin{corollary} 
\label{thm:values}
    Let $F$ be the smallest field over $\Q$ in which the newform Dedekind sum  $\dsum$ takes values. Then $F=F_{\chi_1,\chi_2}$.
\end{corollary}
The acquainted reader will see that Theorem \ref{thm:newknopp} and Proposition \ref{independence} are consequences of the Eichler-Shimura isomorphism. We do not make use of the Eichler-Shimura isomorphism in this paper to minimize prerequisite knowledge and due to the apparent lack of accessible discussion of the Eisenstein part in the literature.
\section{Theorem \ref{thm:newknopp} and Corollaries}
\subsection{Preliminaries}
Our proof of the identity comes from the action of the Hecke operators on the newform Dedekind sums (which we will now call Dedekind sums for brevity) and the weight zero Eisenstein series from which they are derived.

\begin{definition}
Let $\chi_1,\chi_2$ be primitive Dirichlet characters modulo $q_1$ and $q_2$ (respectively) such that $\chi_1\chi_2(-1) = 1$. The completed weight zero newform Eisenstien series is defined as
    \begin{equation*}
        E^*_{\chi_1,\chi_2}(z,s)=\frac{(q_2/\pi)^s}{\tau(\chi_2)}\Gamma(s)L(2s,\chi_1\chi_2)E_{\chi_1,\chi_2}(z,s),
    \end{equation*}
where $\tau(\chi)$ is the Gauss sum of a character $\chi$, $\Gamma(s)$ is the gamma function, $L(s,\chi)$ is a Dirichlet L-function with character $\chi$, and $E_{\chi_1,\chi_2}(z,s)$ is the weight zero newform Eisenstein series:

\begin{equation*}
    E_{\chi_1,\chi_2}(z,s) = \frac{1}{2}\sum_{(c,d)=1}\frac{(q_2y)^s\chi_1(c)\chi_2(d)}{|cq_2z+d|^{2s}}, \: Re(s) > 1.
\end{equation*}
\end{definition} 

\noindent The completed series has a Fourier expansion (see \cite{Calcs}):
    \begin{equation*}
        2\sqrt{y}\sum_{n\neq0}\lambda_{\chi_1,\chi_2}(n,s)\exp(2\pi i n x)K_{s-\frac{1}{2}}(2\pi|n|y).
    \end{equation*}
Here, $K_{\nu}$ is the K-Bessel function and 
    \begin{equation*}
        \lambda_{\chi_1,\chi_2}(n,s) = \chi_2(\sgn(n))\sum_{ad=|n|}\chi_1(a)\oo{\chi_2}(b)\bigg(\frac{b}{a}\bigg)^{s-\frac{1}{2}}.
    \end{equation*}

\noindent Importantly, $E^*_{\chi_1\chi_2}(z,s)$ is an automorphic form on the congruence subgroup $\Gamma_0(q_1q_2)$ with central character $\psi = \chi_1\oo{\chi_2}.$ At $s = 1$ we have the decomposition 
\begin{equation*}
    E^*_{\chi_1\chi_2}(z,1) = f_{\chi_1,\chi_2}(z) + \chi_2(-1)\oo{f}_{\oo{\chi_1},\oo{\chi_2}}(z),
\end{equation*}
where
\begin{equation*}
    f_{\chi_1,\chi_2} = \sum_{n=1}^{\infty}\frac{\lambda_{\chi_1,\chi_2}(n,1)}{\sqrt{n}}\exp(2\pi i n z).
\end{equation*}
\noindent The original definition of the Dedekind sum is as follows:
\begin{definition}[\cite{SVY}]
Let $\chi_1,\chi_2$ be primitive Dirichlet characters modulo $q_1,q_2$ with $q_1,q_2>1$ and $\chi_1\chi_2(-1)=1$. Then for $\gamma\in\Gamma_0(q_1q_2)$ we define $\dsum(\gamma)$ by
\label{def:original}
    $$\dsum(\gamma,z)=\dsum(\gamma) = \frac{\tau(\overline{\chi_1})}{\pi i}( f_{\chi_1,\chi_2}(\gamma z) - \psi(\gamma)f_{\chi_1,\chi_2}(z)).$$
\end{definition}

\begin{remark}
The formula given in Definition \ref{def:finite} is deduced as a theorem in [SVY20]. 
\end{remark}
A brief argument in \cite{SVY} shows $f_{\chi_1,\chi_2}(\gamma z) - \psi(\gamma)f_{\chi_1,\chi_2}(z)$ is both holomorphic and antiholomorphic, and thus $\dsum$ is independent of $z$, motivating the notation of Definition \ref{def:original}. Additionally, Definition \ref{def:finite} combined with Proposition \ref{crossed} and the periodicity of $f_{\chi_1,\chi_2}$ show that $\dsum(\gamma)$ is unchanged under translation, and is determined by only the left column of $\gamma$.

To define the Hecke operators, we compile some results from \cite{Iwan}:
\begin{lemma} \label{cores}
For  
$$\Delta^N_n=\bigg\{\abd : ad = n,\, (a,N) = 1,\, 0 \leq b < d\bigg\},$$
there exists a correspondence between $\Delta^{N}_n \cross \Gamma_0(N)$ and $\Gamma_0(N) \cross \Delta^{N}_n$.
Specifically,
$$\abd
\begin{bmatrix}
        h & * \\
        k & l
     \end{bmatrix}=
     \begin{bmatrix}
        h' & * \\
        k' & l'
     \end{bmatrix}
     \begin{bmatrix}
        a' & b' \\
        0 & d'
     \end{bmatrix}
     $$
    where
    $$h'=\frac{ah + bk}{(ah + bk,dk)},\qquad
    k'=\frac{dk}{(ah + bk,dk)}.$$
\end{lemma}

We now give the definition of the Hecke operators on $H^1(\Gamma_0(N),\C^{\psi}).$
\begin{definition}
\label{def:hecke}
For $n\in\mathbb{N}$ and $\varphi \in H^1(\Gamma_0(N),\C^{\psi})$, the Hecke operator $T_n^{\psi}$ acts on $\varphi$ by
\begin{equation*}
    T_n^{\psi}(\varphi)(\gamma) = \frac{1}{\sqrt{n}}\sum_{\substack{ad=n\\(a,n)=1}}\psi(a)\sum_{b\Mod{d}}
                          \varphi\bigg(\abd\gamma\abdp^{-1}\bigg),
\end{equation*}
where $\abd$ and $\abdp$ are elements of $\Delta^{N}_n$, and have the same relations to $\gamma=\hkl$ as in Lemma \ref{cores}. We note the normalization chosen follows the conventions of \cite{Calcs}.
\end{definition}
The Hecke operators arise from defining double coset operators on $H^1(\Gamma_0(N),\C^{\psi})$ (see Section \ref{sec:coset}) and act linearly on $H^1(\Gamma_0(N),\C^{\psi})$. The Hecke operators are more well known in the context of modular forms and periodic functions on the upper half plane. We mention that the operator $T_n^{\psi}$ acts on weight zero periodic functions by
\begin{equation*}
    T_n^{\psi}(f)(z) = \frac{1}{\sqrt{n}}\sum_{\substack{ad=n\\(a,n)=1}}\psi(a)\sum_{b\Mod{d}} f\bigg(\frac{az + b}{d}\bigg).
\end{equation*}
\subsection{Proof of Theorem \ref{thm:newknopp}}
The identity will follow from the application of $T_n^{\psi}$ to $\dsum$. Specifically, we use Definition \ref{def:original} and calculate with $z=\infty$:

\begin{align}
    T_n^{\psi} \dsum(\gamma) &= \frac{\tau(\overline{\chi_1})}{\sqrt{n}\pi i}\bigg[\sum_{ad=n}\psi(a)\sum_{b\Mod{d}} f_{\chi_1,\chi_2}\bigg(\abd\gamma\abdp^{-1}\infty\bigg) - \psi(\gamma') f_{\chi_1,\chi_2}(\infty)\bigg] \nonumber \\
    &= \frac{\tau(\overline{\chi_1})}{\sqrt{n} \pi i}\sum_{\substack{ad=n\\(a,n)=1}}\psi(a)\sum_{b\Mod{d}} f_{\chi_1,\chi_2}\bigg(\abd\gamma \, \infty \bigg)\qquad (\text{as}\; f_{\chi_1,\chi_2}(\infty) = 0). \label{eq:tnf}
\end{align}
We see that \eqref{eq:tnf} is exactly $\frac{\tau(\overline{\chi_1})}{\pi i}T_n^{\psi}(f_{\chi_1,\chi_2})(\gamma\infty)$.

As noted in \cite{Calcs}, the completed weight zero newform Eisenstein series $E^*_{\chi_1\chi_2}(z,s)$ is an eigenfunction of the Hecke operators:
\begin{equation} \label{eq:2}
    T_n^{\psi} (E^*_{\chi_1\chi_2})(z,s) = \lambda_{\chi_1,\chi_2}(n,s)E^*_{\chi_1\chi_2}(z,s).
\end{equation}
When we apply \eqref{eq:2} at $s=1$ we can deduce
\begin{equation} \label{eq:3}
    T_n^{\psi} (f_{\chi_1\chi_2})(z) = \lambda_{\chi_1,\chi_2}(n,1) f_{\chi_1\chi_2}(z),
\end{equation}
since the Hecke operators preserve holomorphicity.

Therefore when we combine \eqref{eq:tnf} and \eqref{eq:3} we have
\begin{align*}
    \frac{\tau(\overline{\chi_1})}{\pi i}\sum_{\substack{ad=n\\(a,n)=1}}\psi(a)\sum_{b\Mod{d}} f_{\chi_1,\chi_2}(\abd\gamma \, \infty)
    &= \frac{\tau(\overline{\chi_1})}{\pi i} T_n^{\psi}(f_{\chi_1,\chi_2})(\gamma \infty)\\
    &= \frac{\tau(\overline{\chi_1})}{\pi i} \lambda_{\chi_1,\chi_2}(n,1)f_{\chi_1,\chi_2}(\gamma \infty)\\
    &= \lambda_{\chi_1,\chi_2}(n,1)\dsum(\gamma).  
\end{align*}

\noindent Next we need to justify the arguments of $\dsum$ in Theorem \ref{thm:newknopp}. Note that we can extend Definition \ref{def:finite} for any $h \in \Z, k \equiv 0 \, (q_1q_2)$, and we have the following proposition from \cite{DG}:
\begin{proposition}
\label{scale}
Let $h$ and $k$ be coprime integers with $q_1q_2 | k$. Then for all positive integers $\alpha$
$$
    \dsum(\alpha h,\alpha k) = \dsum(h,k).
$$
\end{proposition}
\noindent Then for $h$ and $k$ coprime with $q_1q_2|k$, we see by Lemma \ref{cores} that
\begin{equation*}
    T_n^{\psi}\dsum(h,k) = \frac{1}{\sqrt{n }}\sum_{ad=n}\psi(a)\sum_{b\Mod{d}}\dsum(h',k'),
\end{equation*}
and an application of Proposition \ref{scale} gives $\dsum(h',k') = \dsum(ah+bk,dk)$. Proposition \ref{scale} also removes the coprimality condition on $h$ and $k$.
Finally, we see $\sqrt{n}\lambda_{\chi_1,\chi_2}(n,1) = \rho_{\chi_1,\chi_2}(n)$, proving the identity. 

\begin{remark}
We have that $\rho_{\chi_1,\chi_2}(n)$ are the Fourier coefficients of $E_{2,\chi_1 \chi_2}$, the holomorphic newform Eisenstein series of weight $2$ (see \cite{DS}).
\end{remark}

\subsection{Proof of Proposition \ref{independence}}
We start with the case where all sums have the same central character, and show independence by the Hecke operators.
\begin{lemma} \label{lem:25}
Let $\{(\chi_{k},\chi_k')\}_{k=1}^r$ be a set of distinct ordered pairs of characters modulo $q, q'$. Then there exists a positive integer $n$ where $\rho_{\chi_i,\chi_i'}(n) \neq \rho_{\chi_j,\chi_j'}(n)$ for some $i \neq j$.
\end{lemma}
\begin{proof}
Assume, for the sake of contradiction, that $\rho_{\chi_{i},\chi_i'}(n) = \rho_{\chi_{j},\chi_j'}(n)$ for all $n$ and all $i$ and $j.$ Fix $i,j$ where $i\neq j$. Then for any $p$ prime, we have 
\begin{equation}
\label{eq:6}
\chi_i(p)-\chi_j(p)=p[\oo{\chi_j'}(p)-\oo{\chi_i'}(p)].
\end{equation}
Note that since both the left and right hand sides of \eqref{eq:6} are algebraic integers, their field norms are rational integers. Additionally \eqref{eq:6} implies the norm of the left hand side is divisible by $p$. Then for each residue class coprime to $qq'$ we may use Dirichlet's theorem on primes in arithmetic progressions, along with the fact that the difference of Dirichlet characters can assume finitely many values, to pick sufficiently large primes in that class which force equality between $\chi_i$ and $\chi_j$, and $\chi_i'$ and $\chi_j'$ (we briefly note the use of Dirichlet's theorem can be forgone in favor of an extended argument). This contradicts the distinctness of our original set. 
\end{proof}
Again, by way of contradiction, assume $\{S_{\chi_{k},\chi_k'}\}_{k=1}^r$ is a minimal linearly dependent set of Dedekind sums of central character $\psi$. By Lemma \ref{lem:25}, let $n$ be such that $\rho_{\chi_{i},\chi_i'}(n) \neq \rho_{\chi_{j},\chi_j'}(n)$ for some $i$ and $j$. We apply $\sqrt{n}T_n^{\psi}$ to this combination, which by Theorem \ref{thm:newknopp} scales each term. After relabeling to make $i=1,$ we have
\begin{align*}
    0 &= \sum_{k=1}^r c_k S_{\chi_k,\chi_k'} \:\: (c_k \in \C - \{0\})\\
    &= \sum_{k=1}^r c_k(\rho_{\chi_k,\chi_k'}(n)-\rho_{\chi_1,\chi_1'}(n)) S_{\chi_{k},\chi_k'}
    = \sum_{k=2}^r c_k(\rho_{\chi_k,\chi_k'}(n)-\rho_{\chi_1,\chi_1'}(n)) S_{\chi_{k},\chi_k'},
\end{align*}
which gives a smaller linearly dependent set, contradicting minimality.\\

Independence for newform sums of differing central character is a consequence of the following:
\begin{proposition}
\label{decomp}
We have the decomposition of the space of homomorphisms $Hom(\Gamma_1(N),\C):$
$$Hom(\Gamma_1(N),\C) = \bigoplus_{\psi} H^1(\Gamma_0(N),\C^{\psi}),$$
where $H^1(\Gamma_0(N),\C^{\psi})$ has the group action by Dirichlet character $\psi$ modulo $N$:
$$ \gamma z = \psi(\gamma)z$$
for z $\in \C.$
\end{proposition}

\noindent We defer a proof of Proposition \ref{decomp} to Section \ref{sec:decompf} as it is routine in nature.

\subsection{Proof of Corollary \ref{thm:values}}
First we recall the definition of the Galois action defined in \cite{NRY}.
\begin{definition}
Let $K=Q(\zeta_{lcm(q_1,q_2)})$ be the $lcm(q_1,q_2)^{th}$ cyclotomic field and $\sigma \in Gal(K/\Q).$ Then by Definition \ref{def:finite}, the action of $\sigma$ on $\dsum$ by evaluation gives
    \begin{equation*}
        \sigma\dsum=S_{\chi^{\sigma}_1 \chi^{\sigma}_2}.
    \end{equation*}
\end{definition}
Again, let $F$ be the smallest field containing $\dsum(\Gamma_0(q_1q_2))$. First, we see by definition that $F\subseteq F_{\chi_1,\chi_2}$. For the reverse inclusion, pick $\sigma \in Gal(K/F).$ By the Galois action we have the equality 
$$\dsum=S_{\chi^{\sigma}_1 \chi^{\sigma}_2}.$$
Using Proposition \ref{independence}, we must have that $\chi_1=\chi_1^{\sigma}$ and $\chi_2=\chi_2^{\sigma}$. So $F_{\chi_1,\chi_2} \subseteq F$.

\begin{corollary}
$\dsum$ takes rational values if, and only if, $\chi_1$ and $\chi_2$ are quadratic characters.
\end{corollary}

\section{Proof of Theorem \ref{thm:structure}}
We begin with two lemmas, the first being a consequence of Schreier's Lemma:
\begin{lemma}[\cite{schreier-lemma}]
\label{schrier}
Every finite index subgroup of a finitely generated group is finitely generated.
\end{lemma}
\begin{lemma}
The image $\dsum(\Gamma_1(q_1q_2))$ is a free abelian group.
\end{lemma}
\begin{proof}
First note that since $\dsum \in \text{Hom}(\Gamma_1(q_1q_2),\mathbb{C})$, we must have that $\dsum(\Gamma_1(q_1q_2))$ is a torsion-free group. 
By Lemma \ref{schrier}, since $\Gamma_1(q_1q_2)$ is a subgroup of finite index in $SL_2(\Z)$, it has a finite generating set. Let $\{\gamma_i\}_{i=1}^r$ be a generating set of $\Gamma_1(q_1q_2)$. Then $\{\dsum(\gamma_i)\}_{i=1}^r$ is a generating set for $\dsum(\Gamma_1(q_1q_2)).$
By the structure theorem of abelian groups, $\dsum(\Gamma_1(q_1q_2))$ must be free.
\end{proof}
We next bound the rank of $\dsum(\Gamma_1(q_1q_2))$ from both above and below by $[F_{\chi_1,\chi_2}:\Q]$, beginning with the upper bound.
Recall that the Dedekind sum takes values in the number field $F_{\chi_1,\chi_2}$, which is the fraction field of the ring of algebraic integers $\mathcal{O}_{F_{\chi_1,\chi_2}}$. Using $\{\gamma_i\}_{i=1}^r$ as a generating set of $\Gamma_1(q_1q_2)$, let
$$d = \prod_{i} b_i, \quad \text{ where } \dsum(\gamma_i) = \frac{a_i}{b_i},\quad a_i,b_i\in \mathcal{O}_{F_{\chi_1,\chi_2}}.$$
We then have
$$\dsum(\Gamma_1(q_1q_2)) \subseteq \frac{1}{d}\mathcal{O}_{F_{\chi_1,\chi_2}}.$$
The rank of $\mathcal{O}_{F_{\chi_1,\chi_2}}$ over $\Z$ is precisely $[F_{\chi_1,\chi_2}:\Q]$, showing the upper bound.\\

Next, suppose for contradiction that $\dsum(\Gamma_1(q_1q_2)) = \bigoplus_{i=1}^r \alpha_i\Z$, where $r < [F_{\chi_1,\chi_2}:\Q]=n,\, \alpha_i \in F_{\chi_1,\chi_2}$. Consider the $n$ distinct Dedekind sums
$$\{S_{\chi_1^{\sigma}\chi_2^{\sigma}}|\sigma \in Gal(F_{\chi_1,\chi_2}/\Q)\}.$$
Clearly, $S_{\chi_1^{\sigma}\chi_2^{\sigma}}(\Gamma_1(q_1q_2)) = \bigoplus_{i=1}^r \alpha_i^{\sigma}\Z$. We then construct the matrix $(\alpha_i^{\sigma_j})_{ij}$ which must have a nontrivial kernel by its dimension. This contradicts the linear independence of Dedekind sums, completing the proof.

\begin{remark}
This argument carries over without changes to $\dsum(\Gamma_0(q_1q_2))$ in the case of $\chi_1\oo{\chi_2} = \boldsymbol{1}$.
\end{remark}

When $\chi_1\oo{\chi_2}$ is not trivial, we use that $\Gamma_0(q_1q_2) \rhd \Gamma_1(q_1q_2)$ and Proposition \ref{crossed} to state that
$$\dsum(\gamma) \in \dsum(\beta_i) + \bigoplus_{i=1}^d \alpha_i\Z,$$
where $\gamma \in \Gamma_0(q_1q_2)$ and $\beta_i$ is the coset representative of $\gamma$ in $\Gamma_0(q_1q_2)/\Gamma_1(q_1q_2).$

\begin{remark}
One method of determining a basis of this lattice can be derived from \cite{TW}, which outlines a process to compute the image of the generating set of $\Gamma_1(q_1q_2)$.
\end{remark}

\section{The Cohomological Aspect}
\subsection{Discussion of $H^1(G,M)$} \label{sec:discussion}
When we say the Dedekind sums satisfy the crossed homomorphism property, we need to refine what we mean.
\begin{definition}
Let $G$ be a group and $M$ an abelian group on which $G$ acts compatibly with the additive structure of $M$.
Then we denote by $Z^1(G,M)$ the space of crossed homomorphisms. That is, maps $\varphi: G \longrightarrow M$ satisfying:
$$\varphi(gh) = \varphi(g) + g\varphi(h).$$
\end{definition}
\noindent We then construct the first cohomology group, $H^1(G,M)$, as a quotient of $Z^1(G,M).$
\begin{definition}
We set $H^1(G,M) = Z^1(G,M)/B^1(G,M)$ where $B^1(G,M)$ is the subgroup generated by principal crossed homomorphisms. These are crossed homomorphisms $\varphi$ such that
\begin{equation*}
    \varphi(g) = gm-m,
\end{equation*}
where $m \in M.$
\end{definition}

\noindent In light of Proposition \ref{decomp}, we view the Dedekind sums as elements of $H^1(\Gamma_0(q_1q_2),\C^{\psi})$, and not elements of $Z^1(\Gamma_0(q_1q_2),\C^{\psi})$. In fact, quotienting out by principal crossed homomorphisms is necessary to show Lemma \ref{lem:split} and the well-definition of the projections in Definition \ref{def:projs}.

\subsection{Proof of Proposition \ref{decomp}}\label{sec:decompf}
We begin by defining the family of projections we use to prove the decomposition.
\begin{definition}
\label{def:projs}
We define the projection $\pi_{\psi}:Hom(\Gamma_1(N),\C) \longrightarrow H^1(\Gamma_0(N),\C^{\psi})$ by
\begin{equation*}
    \pi_{\psi} : \varphi \mapsto \frac{1}{\phi(n)} \sum_{j \in (\Z/N\Z)^{\cross}} \oo{\psi}(\beta_j)   \varphi(\beta_j \gamma \beta_{\sigma_{\gamma}(j)}^{-1}),
\end{equation*}
where $\beta_j$ are the right coset representatives of $\Gamma_0(N)/\Gamma_1(N)$, and $\beta_{\sigma_{\gamma}(j)}$ is the unique coset representative where $\beta_j \gamma \beta_{\sigma_{\gamma}(j)}^{-1} \in \Gamma_1(N)$ for $\gamma \in \Gamma_0(N).$
\end{definition}

\begin{remark}
Well definition of the projections, which are examples of a kind of map called co-restriction, is shown in a near identical way  to the proof of Lemma \ref{lem:split}. By direct computation we have $\sigma_{\gamma}$ is a permutation of $(\Z/N\Z)^{\cross}$, and that $\sigma_{\gamma_1 \gamma_2}(j) = \sigma_{\gamma_2}(\sigma_{\gamma_1}(j))$.
\end{remark}

We demonstrate that this family gives the desired decomposition in the next three lemmas.

\begin{lemma}
 We have $\pi_{\psi}(Hom(\Gamma_1,\C)) \subseteq H^1(\Gamma_0(N),\C^{\psi}).$
\end{lemma}
\begin{proof}
Let $\varphi \in Hom(\Gamma_1,\C))$. Then for $\gamma_1,\gamma_2 \in \Gamma_0(N)$ we have
\begin{align*}
    \pi_{\psi}(\varphi)(\gamma_1\gamma_2) &=\projsum \oo{\psi}(\beta_j) \varphi(\beta_j \gamma_1 \gamma_2 \beta_{\sigma_{\gamma_1 \gamma_2}(j)}^{-1}) \\
    &= \projsum \oo{\psi}(\beta_j)[\varphi(\beta_j \gamma_1 \beta_{\sigma_{\gamma_1}(j)}^{-1}) + \psi(\beta_j \gamma_1 \beta_{\sigma_{\gamma_1}(j)}^{-1}) \varphi(\beta_{\sigma_{\gamma_1}(j)} \gamma_s \beta_{\sigma_{\gamma_1 \gamma_2}(j)}^{-1})]\\
    &= \projsum \oo{\psi}(\beta_j) \varphi(\beta_j \gamma_1 \beta_{\sigma_{\gamma_1}(j)}^{-1}) + \psi(\gamma_1) \projsum\oo{\psi}(\beta_{\sigma_{\gamma_1}(j)}) \varphi(\beta_{\sigma_{\gamma_1}(j)} \gamma_2 \beta_{\sigma_{\gamma_2}(\sigma_{\gamma_1}(j))}^{-1})\\
    &= \pi_{\psi}(\varphi)(\gamma_1) + \psi(\gamma_1)\pi_{\psi}(\varphi)(\gamma_2).
\end{align*}
\end{proof}

\begin{lemma}
\label{lem:split}
The projection $\pi_{\psi}$ is the identity on $H^1(\Gamma_0(N),\C^{\chi})$ if, and only if, $\psi=\chi$, and is the zero map otherwise.
\end{lemma}
\begin{proof}
Let $\varphi \in H^1(\Gamma_0(N),\C^{\chi})$ and $\gamma \in \Gamma_0(N)$. We have

\begin{align}
    \pi_{\psi}(\varphi)(\gamma) &= \projsum \oo{\psi}(\beta_j) \varphi(\beta_j \gamma \beta_{\sigma_{\gamma}(d)}^{-1}) \nonumber \\
    &= \projsum \oo{\psi}(\beta_j) [\varphi(\beta_j) + \chi(\beta_j)\varphi(\gamma) + \chi(\beta_j \gamma)\varphi(\beta_{\sigma_{\gamma}(d)}^{-1})] \nonumber \\
    &= \projsum \oo{\psi}(\beta_j)\chi(\beta_j)\varphi(\gamma) + \projsum \oo{\psi}(\beta_j) [\varphi(\beta_j) + \chi(\beta_j \gamma)\varphi(\beta_{\sigma_{\gamma}(d)}^{-1})] \label{eq:8}.
\end{align}
Using the orthogonality relations of Dirichlet characters, the first term in \eqref{eq:8}
equals $\varphi(\gamma)$ when $\psi = \chi$ and otherwise vanishes. By the substitution $\beta_j \gamma \equiv \beta_{\sigma_{\gamma}(j)} \pmod{\Gamma_1(N)}$, and the fact that $\chi(\beta)\varphi(\beta^{-1}) = -\varphi(\beta)$, the second term is equal to the principal crossed homomorphism
$$
-\psi(\gamma) \bigg(\projsum \oo{\psi}(\beta_j)\varphi(\beta_j) \bigg) + \projsum\oo{\psi}(\beta_j)\varphi(\beta_j)
$$
which vanishes in $H^1(\Gamma_0(N),\C^{\psi})$.
\end{proof}

\begin{lemma}
The map $\sum_{\psi} \pi_{\psi}$ is the identity on $Hom(\Gamma_1(N),\C).$
\end{lemma}
\begin{proof}
Let $\varphi \in Hom(\Gamma_1(N),\C).$ Then we have 
\begin{align*}
    \sum_{\psi} \pi_{\psi}(\varphi)(\gamma) &= \sum_{\psi} \projsum \oo{\psi}(\beta_j) \varphi(\beta_j \gamma \beta_{\sigma_{\gamma}(j)}^{-1}) \\
    &= \projsum \varphi(\beta_j \gamma \beta_{\sigma_{\gamma}(j)}^{-1}) \sum_{\psi} \oo{\psi}(\beta_j)\\
    &= \varphi(\gamma).
\end{align*}
\end{proof}

\subsection{Double Coset Operators}\label{sec:coset}
As in the case of modular forms (see \cite{DS}), a double coset operator can be defined from $H^1(\Gamma_1, M)$ to $H^1(\Gamma_2,M)$ for congruence subgroups $\Gamma_1$ and $\Gamma_2$.
\begin{definition}
For $\alpha \in GL_2^+(\Q)$, and congruence subgroups $\Gamma_1$ and $\Gamma_2$, define the double coset operator $\Gamma_1 \alpha \Gamma_2$ on $\varphi \in H^1(\Gamma_1,*)$ by
\begin{equation*}
    \varphi|_{\Gamma_1 \alpha \Gamma_2}(\gamma) = \sum_{\beta_i} \widetilde{\beta_i}\varphi(\beta_i \gamma \beta_{\sigma_{\gamma}(j)}^{-1}).
\end{equation*}
Here $\gamma \in \Gamma_2$, $\beta_i$ are the orbit representatives of $\Gamma_1 \backslash \Gamma_1 \alpha \Gamma_2$,  $\widetilde{\beta_i} = det(\beta_i)\beta_i^{-1}$, and $\beta_{\sigma_{\gamma}(j)}$ is the unique orbit representative such that $\beta_i \gamma \beta_{\sigma_{\gamma}(j)}^{-1} \in \Gamma_1$.
\end{definition}
Observe that this requires a definition for the action of the orbit representatives on the common group module $M$. We conclude with a brief motivitation for the defintion of the Hecke operators acting on $H^1(\Gamma_0(N),\C^{\psi})$ with this new context.

The most important specification to make for double coset operators is defining the action of the orbit representatives. Action by Dirichlet character $\psi$ on $\C$ can be naturally extended to the monoid $ \mathcal{M} \coloneqq \bigg\{\gamma = \mat \in M_2(\Z) : N\,|\,c \bigg\}$ by $\gamma z = \psi(\gamma)z = \psi(d)z$. From \cite{Iwan}, we have that $\Delta_n^N$ is a complete set of orbit representatives for $T_n^{\psi}$, all of which lie in $\mathcal{M}$. Their action gives the multiplication by the character $\psi$ given in Definition \ref{def:hecke}.

\section{Acknowledgements}
This work was conducted during the summer of 2022 at an REU at Texas A\&M University. The author thanks the Department of Mathematics at Texas A\&M and the NSF (DMS-2150094) for supporting the REU.
The author would like to thank Matthew Young (Texas A\&M University) for guiding this research project, and Agniva Dasgupta (Texas A\&M University) for his lectures and feedback during the REU.

\bibliography{main}
\bibliographystyle{alpha}

\Addresses

\end{document}